\documentclass{amsart}

\usepackage{amsmath}
\usepackage{amsfonts}
\usepackage{amssymb}
\usepackage{amsthm}
\usepackage{comment}
\usepackage{epsfig}
\usepackage{psfrag}
\usepackage{mathrsfs}
\usepackage{amscd}
\usepackage{color}
\usepackage[all]{xy}
\usepackage{rotating}
\usepackage{lscape}
\usepackage{amsbsy}
\usepackage{verbatim}
\usepackage{moreverb}
\usepackage{mathdots}
\usepackage{setspace}

\usepackage[margin=1in]{geometry}

\theoremstyle{definition}
\newtheorem{theorem}{Theorem}[subsection]
\newtheorem{prop}[theorem]{Proposition}
\newtheorem{lemma}[theorem]{Lemma}
\newtheorem{corollary}[theorem]{Corollary}

\newtheorem{definition}[theorem]{Definition}
\newtheorem{notation}[theorem]{Notation}

\newtheorem{construction}[theorem]{Construction}
\newtheorem{remark}[theorem]{Remark}
\newtheorem{example}[theorem]{Example}

\def\cc{\mathbf C}\def\cd{\mathbf D}
\def\ce{\mathcal E}\def\cf{\mathcal F}

\def\cP{\mathcal P}

\newcommand{\calO}{{\mathcal O}}

\newcommand{\set}{{\mathrm{Set}}}
\newcommand{\Deltab}{{\mathbf \Delta}}

\begin{document}
\title{2-fold complete Segal operads}
\author{Sanath K. Devalapurkar}
\maketitle
\begin{abstract}
We define the notion of Cartesian 2-fibrations, and prove a weak analogue of straightening. Using Barwick's notion of operator categories and the notion of a Cartesian 2-fibration, we extend the notion of $\infty$-operads to the $(\infty,2)$-categorical context. These $(\infty,2)$-operads are used to define and study 2-rings, which are analogues of rings (and ring spectra) in the $\infty$-categorical context.
\end{abstract}
\section{Introduction}
Let $\cf_\ast$ denote the category of pointed finite sets ($\cf$ denotes the category of finite sets). Recall that a symmetric monoidal $\infty$-category $\cc^\otimes$ is an $\infty$-operad whose structure map $\cc^\otimes\to\mathrm{N}(\cf_\ast)$ is a coCartesian fibration. This makes sense because of straightening, which says, in this case, that this is equivalent to a map from $\mathrm{N}(\cf_\ast)$ to the $\infty$-category of $\infty$-categories; since $\mathrm{N}(\cf_\ast)$ is equivalent to the $\mathbf{E}_\infty$-operad, this is simply an $\mathbf{E}_\infty$-algebra object in $\mathrm{Cat}_\infty$, i.e., a symmetric monoidal $\infty$-category.

Consequently, in order to define symmetric monoidal $(\infty,2)$-categories, it is essential to define the $(\infty,2)$-categorical version of an $\infty$-operad, i.e., an $(\infty,2)$-operad. There are multiple models of $\infty$-operads, including Moerdijk-Weiss' dendroidal sets, Lurie's $\infty$-operads, and Barwick's complete Segal operads. Since we are generalizing to $(\infty,2)$-category theory through scaled simplicial sets, which generalizes many simplicial constructions, the complete Segal operad model has a straightforward generalization to the notion of a $2$-fold complete Segal operad. We formalize the $(\infty,2)$-categorical analogues $\mathbf{E}^2_k$ of the little cubes operads, and prove that the ordinary $\mathbf{E}_k$-operads are subsimplicial sets of these $(\infty,2)$-operads.

\section{Scaled simplicial sets}
\subsection{Model structure}
\begin{definition}
Let $X$ be a simplicial set. A scaled simplicial set is a pair $(X,\Gamma)$, where $\Gamma$ is a collection of $2$-simplices of $X$ such that the collection $\mathrm{deg}_2(X)$ of degenerate $2$-simplices of $X$ is a subcollection of $\Gamma$. The edges in $\Gamma$ are called thin $2$-simplices.
\end{definition}
There are two canonical ways of presenting any simplicial set $X$ as a scaled simplicial set. We can form a scaled simplicial set $X^\flat$ (resp. $X^\sharp$) where only the degenerate $2$-simplices (resp. all of the $2$-simplices) of $X$ are thin.
\begin{definition}
Suppose $(X,\Gamma)$ and $(Y,\Omega)$ are scaled simplicial sets. A \textit{map} from $(X,\Gamma)$ to $(Y,\Omega)$ is a map $f:X\to Y$ such that $f(\Gamma)\subseteq \Omega$.
\end{definition}
\begin{remark}
The collection of scaled simplicial sets and maps between them forms a category, $\set_{\Deltab}^\mathrm{sc}$.
\end{remark}
\begin{notation}
We will not specify the collection of thin $2$-simplices, and we will usually write $X$ instead of $(X,\Gamma)$ if there is no risk of confusion.
If $X$ is a simplicial set, we will write $(\set_{\Deltab}^\mathrm{sc})_{/X}$ to denote $(\set_{\Deltab}^\mathrm{sc})_{/X^\sharp}$.
\end{notation}
\begin{remark}
Let $\set_{\Deltab}$ denote the model category of simplicial sets with the Kan model structure. A map is called \textit{anodyne} if it has the right lifting property with respect to every fibration. In other words, the collection of anodyne maps is simply the collection of trivial cofibrations in $\set_{\Deltab}$. A map $X\to {\Delta^0}$ is a fibration if it has the right lifting property with respect to the horn inclusions $\Lambda^n_i\hookrightarrow\Delta^n$ for $0\leq i \leq n$. Thus the horn inclusions $\Lambda^n_i\hookrightarrow\Delta^n$ for $0\leq i\leq n$ are generating trivial cofibrations for the Kan model structure. As this discussion shows, anodyne maps play a very important role in ordinary homotopy theory. This important notion can be constructed in the setting of marked simplicial sets as well.
\end{remark}
\begin{definition}
The class of marked anodyne maps is the smallest weakly saturated class of maps generated by:
\begin{enumerate}
\item The inclusions $(\Lambda^n_i)^\flat\hookrightarrow(\Delta^n)^\flat$ for $0<i<n$.
\item The inclusion $(\Lambda^n_n,(\mathrm{deg}_1(\Delta^n)\cup\Delta^{\{n-1,n\}})\cap(\Lambda^n_n)_1)\hookrightarrow(\Delta^n,\mathrm{deg}_1(\Delta^n)\cup\Delta^{\{n-1,n\}})$.
\item The inclusion $(\Lambda^2_1)^\sharp\coprod_{(\Lambda^2_1)^\flat}(\Delta^2)^\flat\to(\Delta^2)^\sharp$.
\item The map $K^\flat\to K^\sharp$ for every Kan complex $K$.
\end{enumerate}
\end{definition}
Let us recall the concept of a $p$-Cartesian morphism for an inner fibration $p:X\to S$.
\begin{definition}
Suppose $p:X\to S$ is an inner fibration and $f:x\to y$ an edge in $X$. $f$ is said to be $p$-Cartesian if the map $X_{/f}\to X_{/y}\times_{S_{/p(y)}}S_{/p(f)}$ is a trivial Kan fibration.
\end{definition}
\begin{prop}\label{charmarkedanodynemaps}
\cite[Proposition 3.1.1.6]{highertopos} A map $p:X\to S$ is a marked anodyne map if and only if it has the left lifting property with respect to every map $i: Y\to Z$ in $\set_{\Deltab}^+$ such that:
\begin{enumerate}
\item $i$ is an inner fibration on the underlying simplicial sets.
\item An edge of $Y$ is marked if and only if its image under $i$ is and it is $i$-Cartesian.
\item If $f$ is a map in $Z$ whose target is the image of a $0$-simplice of $Y$, then $f$ can be pulled back via $i$ to a map in $Y$.
\end{enumerate}
\end{prop}
\begin{definition}
The collection of scaled anodyne maps is the smallest weakly saturated class generated by:
\begin{enumerate}
\item The inclusions $(\Lambda^n_i,(\mathrm{deg}_2(\Delta^n)\cup\{\Delta^{\{i-1,i,i+1\}}\})\cap\mathrm{Map}(\Delta^2,\Lambda^n_i))\hookrightarrow(\Delta^n,(\mathrm{deg}_2(\Delta^n)\cup\{\Delta^{\{i-1,i,i+1\}}\})$.
\item Let $T$ denote $\mathrm{deg}_2(\Delta^4)\cup\{\Delta^{\{0,2,4\}},\Delta^{\{1,2,3\}},\Delta^{\{0,1,3\}},\Delta^{\{1,3,4\}}\}\cup\Delta^{\{0,1,2\}}$. The inclusion $(\Delta^4,T)\hookrightarrow(\Delta^4,T\cup\{\Delta^{\{0,3,4\}},\Delta^{\{1,3,4\}}\})$.
\item Let $S$ denote $\mathrm{deg}_2(\Delta^n\coprod_{\Delta^{\{0,1\}}}\Delta^0)\cup\mathrm{im}(\Delta^{\{0,1,n\}})$. For $n>2$, the inclusion $(\Lambda^n_0\coprod_{\Delta^{\{0,1\}}}\Delta^0,S)\hookrightarrow(\Delta^n\coprod_{\Delta^{\{0,1\}}}\Delta^0,S)$.
\end{enumerate}
\end{definition}
The scaled anodyne maps act as generating trivial cofibrations in a particular model structure on $\set_{\Deltab}^\mathrm{sc}$, which we will now develop. The structure of the collection of scaled anodyne maps suggests that this model structure is very similar to the Joyal model structure. It is therefore important (and interesting) to develop a scaled analogue of the simplicial nerve functor, whose construction we will now recall.
\begin{construction}
Let $\cc$ be a $\set_{\Deltab}^+$-enriched category; there is a forgetful functor $\set_{\Deltab}^+\to\set_{\Deltab}$, so we may regard $\cc$ as a simplicial category. Define $\mathrm{N}^\mathrm{sc}(\cc)$ to be the following scaled simplicial set: the underlying simplicial set of $\mathrm{N}^\mathrm{sc}(\cc)$ is the simplicial nerve $\mathrm{N}(\cc)$ of $\cc$, and thin $2$-simplices are defined as follows. Let $\sigma:\Delta^2\to\mathrm{N}(\cc)$ be a $2$-simplex. This can be thought of a diagram in $\mathrm{N}(\cc)$ that shows a morphism $h:X\to Z$ in $\mathrm{N}(\cc)$ as the composition of a map $f\circ g$. Let $\alpha:\Delta^1\to\mathrm{Map}_\cc(X,Z)$ be an edge in $\mathrm{Map}_\cc(X,Z)$ joining $h$ to $f\circ g$. $\mathrm{Map}_\cc(X,Z)$ is a marked simplicial set, and $\sigma$ is thin if and only $\alpha$ is marked in $\mathrm{Map}_\cc(X,Z)$. This construction is functorial in $\cc$, and defines a functor $\mathrm{N}^\mathrm{sc}:\mathrm{Cat}_{\set_{\Deltab}^+}\to\set_{\Deltab}^\mathrm{sc}$. This functor has a left adjoint $\mathfrak{C}^\mathrm{sc}:\set_{\Deltab}^\mathrm{sc}\to\mathrm{Cat}_{\set_{\Deltab}^+}$.
\end{construction}
This construction leads to the following definition.
\begin{definition}
Let $f:X\to Y$ be a map of scaled simplicial sets. It is a \textit{bicategorical equivalence} if $\mathfrak{C}^\mathrm{sc}[f]$ is an equivalence of $\set_{\Deltab}^+$-enriched categories.
\end{definition}
In \cite{goodwillie}, the following two results are proved:
\begin{prop}\label{trivialcofib}
The functor $\mathfrak{C}^\mathrm{sc}$ preserves cofibrations and takes scaled anodyne maps to trivial cofibrations.
\end{prop}
\begin{lemma}
Bicategorical equivalences satisfy the following properties:
\begin{enumerate}
\item Any scaled anodyne map is a bicategorical equivalence.
\item Bicategorical equivalences are closed under pushouts by cofibrations.
\end{enumerate}
\end{lemma}
\begin{remark}
The second part of this lemma suggests the existence of a model structure on $\set_{\Deltab}^\mathrm{sc}$ that satisfies \cite[Proposition A.2.6.13]{highertopos}. Indeed, the above proposition shows that $\mathfrak{C}^\mathrm{sc}$ is a functor that gives a bijection between the collection of bicategorical equivalences and the image of the collection of weak equivalences in $\mathrm{Cat}_{\set_{\Deltab}^+}$ under $\mathfrak{C}^\mathrm{sc}$.
\end{remark}
Lurie has proven (c.f. \cite{goodwillie}) that the last condition of \cite[Proposition A.2.6.13]{highertopos} is satisfied, giving a model structure very similar to the Joyal model structure:
\begin{theorem}\label{scaledmodel}
There is a left proper combinatorial model structure on $\set_{\Deltab}^\mathrm{sc}$ whose weak equivalences are the bicategorical equivalences and cofibrations are the monomorphisms on the underlying simplicial sets. The functors $\mathrm{N}^\mathrm{sc}$ and $\mathfrak{C}^\mathrm{sc}$ form a Quillen adjunction, and more precisely a Quillen equivalence, between $\set_{\Deltab}^\mathrm{sc}$ and $\mathrm{Cat}_{\set_{\Deltab}^+}$.
\end{theorem}
\begin{definition}
An $(\infty,2)$-category is a simplicial set with the right lifting property with respect to the scaled anodyne maps.
\end{definition}
\begin{example}
A fibrant object of $\set_{\Deltab}^\mathrm{sc}$ is an $(\infty,2)$-category.
\end{example}
\begin{remark}
This makes sense because if $X$ is a fibrant object of $\set_{\Deltab}^\mathrm{sc}$, and if $x,y\in X$, then $\mathrm{Map}_X(x,y)$ is a fibrant object of $\set_{\Deltab}^+$, i.e., an $\infty$-category.
\end{remark}
\begin{example}\label{basiclimits}
Consider the model category $\set^+_\Deltab$ of scaled simplicial sets; this is enriched over $\set_\Deltab$. Regarding every mapping simplicial set as a marked simplicial set (via the marking where every $1$-simplex is marked), this becomes a fibrant $\set^+_\Deltab$-enriched category, so one can take the scaled nerve, giving the $(\infty,2)$-category $\mathrm{Cat}_\infty$ of $\infty$-categories.
\end{example}
\begin{example}\label{prototype}
Let $\mathbf{Cat}^\mathrm{st}_\infty$ be the $\set_{\Deltab}^+$-enriched category of stable $\infty$-categories $\cc,\cd$, where $\mathrm{Map}_{\mathbf{Cat}^\mathrm{st}_\infty}(\cc,\cd)$ is defined to be $\mathrm{Fun}^\mathrm{ex}(\cc,\cd)$. Note that for all stable $\infty$-categories $\cc$ and $\cd$, $\mathrm{Map}_{\mathbf{Cat}^\mathrm{st}_\infty}(\cc,\cd)$ is itself a stable $\infty$-category. Define $\mathrm{Cat}^\mathrm{st}_\infty$ to be the scaled nerve $\mathrm{N}^\mathrm{sc}(\mathbf{Cat}^\mathrm{st}_\infty)$ of $\mathbf{Cat}^\mathrm{st}_\infty$. This is a subcategory of $\mathrm{Cat}_\infty$ since the identity functor is exact and a composition of exact functors is also exact.
\end{example}
The Cartesian product of scaled simplicial sets gives $\set^\mathrm{sc}_\Deltab$ the structure of a monoidal model category.
\begin{prop}\label{monoidalmodel}
The Cartesian product of scaled simplicial sets defined as $(X,\ce)\times(X^\prime,\ce^\prime)=(X\times X^\prime,\ce\times\ce^\prime)$ endows $\set^{\mathrm{sc}}_{\Deltab}$ with the structure of a monoidal model category.
\end{prop}
\begin{proof}
It is easy to see that the functor $\times:\set^{\mathrm{sc}}_{\Deltab}\times\set^{\mathrm{sc}}_{\Deltab}\to\set^{\mathrm{sc}}_{\Deltab}$ preserves colimits separately in each variable. We claim that if $i:(X,\ce)\to(X^\prime,\ce^\prime)$ or $j:(Y,\cf)\to(Y^\prime,\cf^\prime)$ is a (trivial) cofibration, then the induced map $i\vee j:(X^\prime\times Y,\ce^\prime\times\cf)\coprod_{(X\times Y,\ce\times\cf)}(X\times Y^\prime,\ce\times\cf^\prime)\to(X^\prime\times Y^\prime,\ce^\prime\times\cf^\prime)$ is also a (trivial) cofibration. If $i$ or $j$ is a cofibration, then it is automatically true that $i\vee j$ is, since the condition that a map is a cofibration depends only on the underlying simplicial set. We now observe that the Cartesian product makes $\set_{\Deltab}$ with the Joyal model structure into a monoidal model category. This proves the claim.
We must now show that if $i$ or $j$ is a trivial cofibration, then so is $i\vee j$. Let us work one variable at a time. Suppose that $i$ is a trivial cofibration and $Y=Y^\prime$. By definition, $\mathfrak{C}^\mathrm{sc}[X]\simeq\mathfrak{C}^\mathrm{sc}[X^\prime]$ and $\mathfrak{C}^\mathrm{sc}[Y]\simeq\mathfrak{C}^\mathrm{sc}[Y^\prime]$. Since $\mathfrak{C}^\mathrm{sc}$ is a left adjoint, it preserves colimits (and pushouts, in particular), so that there is a pushout diagram $\mathfrak{C}^\mathrm{sc}[i\vee j]:\mathfrak{C}^\mathrm{sc}[(X^\prime\times Y,\ce^\prime\times\cf)]\coprod_{\mathfrak{C}^\mathrm{sc}[(X\times Y,\ce\times\cf)]}\mathfrak{C}^\mathrm{sc}[(X\times Y^\prime,\ce\times\cf^\prime)]\to\mathfrak{C}^\mathrm{sc}[(X^\prime\times Y^\prime,\ce^\prime\times\cf^\prime)]$. It now suffices to show that the Cartesian product preserves weak equivalences; this is precisely the content of \cite[Lemma 4.2.6]{goodwillie}. There is now a chain of equivalences $\mathfrak{C}^\mathrm{sc}[(X^\prime\times Y,\ce^\prime\times\cf)]\simeq{\mathfrak{C}^\mathrm{sc}[(X\times Y,\ce\times\cf)]}\simeq\mathfrak{C}^\mathrm{sc}[(X\times Y^\prime,\ce\times\cf^\prime)]\simeq\mathfrak{C}^\mathrm{sc}[(X^\prime\times Y^\prime,\ce^\prime\times\cf^\prime)]$. This shows that if $i$ or $j$ is a trivial cofibration, so is $i\vee j$, which concludes the proof that the functor $\times:\set^{\mathrm{sc}}_{\Deltab}\times\set^{\mathrm{sc}}_{\Deltab}\to\set^{\mathrm{sc}}_{\Deltab}$ is a left Quillen bifunctor.
Every object is cofibrant, and the monoidal structure is closed and symmetric. Hence the Cartesian product endows $\set^{\mathrm{sc}}_{\Deltab}$ with the structure of a monoidal model category.
\end{proof}

\subsection{Fibrations of scaled simplicial sets}
We need a preliminary definition.
\begin{definition}
Let $(Y,\mathbf{Y})$ and $(K,\mathbf{K})$ be scaled simplicial sets. Define $Y\diamond K$ as the alternative join construction of $Y$ and $K$ as simplicial sets with the obvious thin $2$-simplices.
\end{definition}
\begin{remark}
This can be viewed as the coproduct, with the Gray product (from \cite{harpaz}):
$$\overline{Y}\coprod_{Y\times K\times \Delta^{\{0\}}}\overline{Y}\times_\mathrm{gr}\overline{K}\times_\mathrm{gr}\overline{\Delta^1}\coprod_{Y\times K\times\Delta^{\{1\}}}\overline{K},$$
where $\overline{Y}=(Y,\mathrm{deg}_1(Y),\mathbf{Y})$, $\overline{K}=(K,\mathrm{deg}_1(K),\mathbf{K})$, and $\overline{\Delta^1}=(\Delta^1,\mathrm{deg}_1(\Delta^1),\mathrm{deg}_2(\Delta^1))$.
\end{remark}
\begin{definition}
A map $p:X\to S$ is an \textit{inner $2$-fibration} if it has the right lifting property with respect to all scaled anodyne maps.
\end{definition}
\begin{remark}
Any fibration in the ordinary model structure on $\set^\mathrm{sc}_\Deltab$ is an inner $2$-fibration, since every scaled anodyne map is a trivial cofibration. If $p:X\to S$ is an inner $2$-fibration, then the underlying map of simplicial sets is an inner fibration, because it has the right lifting property with respect to the inclusions $\Lambda^n_i\hookrightarrow \Delta^n$ for $0<i<n$. Recall that if $X$ is a simplicial set, the map $\mathrm{Fun}(\Delta^1,X)\to X$ given by evaluating at $0$ (resp. $1$) is a Cartesian (resp. coCartesian) fibration. Thus, if $p:X\to S$ is an inner $2$-fibration such that the underlying map of simplicial sets is a Cartesian fibration, then the composition $\mathrm{Map}((\Delta^1)^\sharp,X)\to X\xrightarrow{p} S$ given by evaluating on $0$ is a Cartesian fibration on the underlying simplicial sets.
\end{remark}
\begin{definition}
Let $p:X\to S$ be an inner $2$-fibration whose underlying map of simplicial sets is a Cartesian fibration. A $2$-simplex $\phi:f\to g$ in $X$ is \textit{$p$-Cartesian} if the following condition holds: let $q:\mathrm{Map}(\Delta^1,X)\to X$ denote the map by evaluation on the initial vertex, i.e., $0$. Then $\phi$, when viewed as a $1$-simplex of $\mathrm{Map}(\Delta^1,X)$, is $p\circ q$-Cartesian. 
\end{definition}
The collection of $p$-Cartesian $2$-simplices of $X$ forms a scaled simplicial set $X^\natural$.
\begin{definition}
A \textit{Cartesian $2$-fibration} of scaled simplicial sets is a map $p:X\to S$ of scaled simplicial sets such that:
\begin{itemize}
\item $p$ is an inner $2$-fibration whose underlying map of simplicial sets is a Cartesian fibration.
\item A thin $2$-simplex $s\to p(y)$ in $S$ can be pulled back to a thin $p$-Cartesian $2$-simplex $x\to y$ in $X$ such that $p(x)=s$.
\end{itemize}
\end{definition}
A map $p:X\to S$ is a \textit{left $2$-fibration} if $p$ is a coCartesian $2$-fibration and $p$ has the right lifting property with respect to the collection of inclusions $(\Lambda^n_0,(\mathrm{deg}_2(\Delta^n)\cup\{\Delta^{\{0,1\}}\})\cap\mathrm{Map}(\Delta^2,\Lambda^n_0))\hookrightarrow(\Delta^n,(\mathrm{deg}_2(\Delta^n)\cup\{\Delta^{\{0,1\}}\})$. This is equivalent to asking that $p:X\to S$ is a coCartesian $2$-fibration and all $2$-simplices are $p$-Cartesian.
\begin{theorem}\label{leftmodel}
There is a model structure on $(\set^\mathrm{sc}_\Deltab)_{/S}$ such that:
\begin{itemize}
\item The cofibrations are the monomorphisms of the underlying simplicial sets.
\item The weak equivalences are those maps $f:X\to Y$ such that the induced map $\mathfrak{C}^\mathrm{sc}[X^\vartriangleleft\coprod_X S]\to\mathfrak{C}^\mathrm{sc}[Y^\vartriangleleft\coprod_Y S]$ is an equivalence of $\set^+_\Deltab$-enriched categories.
\item The fibrations are those maps with the right lifting property with respect to those maps which are both cofibrations and weak equivalences (these are the trivial cofibrations).
\end{itemize}
\end{theorem}
\begin{proof}
The proof is the same as \cite[Proposition 2.1.4.7]{highertopos}. Namely, we have to show that the hypotheses of \cite[Proposition A.2.6.13]{highertopos} are satisfied. By the same argument in \emph{loc. cit.}, the weak equivalences form a perfect class.

Consider a pushout diagram of the form:
\begin{equation*}
\xymatrix{X\ar[r]^p\ar[d]^{i} & Y\ar[d]\\
X^\prime\ar[r]^{p^\prime} & Y^\prime}
\end{equation*}
where $p$ is a weak equivalence and $i$ is a cofibration. We know that $\mathfrak{C}^\mathrm{sc}$ takes such a pushout diagram to a homotopy coCartesian diagram in $\mathrm{Cat}_{\set^+_\Deltab}$ because $\mathrm{Cat}_{\set^+_\Deltab}$ is left proper and \cite[Proposition 3.1.13]{goodwillie} says that $\mathfrak{C}^\mathrm{sc}$ takes monomorphisms to cofibrations.

Proving that any map $p:X\to S$ with the right lifting property with respect to any cofibration is a weak equivalence. As in \cite[Proposition 2.1.4.7]{highertopos}, this reduces to showing that the homotopy $(\Delta^1)^\sharp\times X\to X$ between $\mathrm{id}_X$ and the composition of $p$ with a section is a weak equivalence. Since $h$ has a left inverse, it suffices to prove that the left inverse $X\times\{0\}\subseteq X\times\Delta^1$ is scaled anodyne (because scaled anodyne maps are weak equivalences), which follows from \cite[Proposition 3.1.8]{goodwillie}.
\end{proof}
\begin{lemma}
Every scaled anodyne map and the collection of inclusions $(\Lambda^n_0,(\mathrm{deg}_2(\Delta^n)\cup\{\Delta^{\{0,1\}}\})\cap\mathrm{Map}(\Delta^2,\Lambda^n_0))\hookrightarrow(\Delta^n,(\mathrm{deg}_2(\Delta^n)\cup\{\Delta^{\{0,1\}}\})$ is a trivial cofibration in the above model structure.
\end{lemma}
In particular, every fibration in Theorem \ref{leftmodel} is a left $2$-fibration; therefore, we may call this model structure the left model structure. Suppose $p:X\to S$ is a coCartesian $2$-fibration; then if $s\to s^\prime$ is a map in $S$, there is an induced map of $(\infty,2)$-categories $X_{s^\prime}\to X_{s}$.
\begin{theorem}\label{straight}
There is a map of $(\infty,2)$-categories from the $(\infty,2)$-category of coCartesian fibrations over a scaled simplicial set $(S,\Gamma)$ to the $(\infty,2)$-category of functors from $(S,\Gamma)$ to the $(\infty,2)$-category $\mathrm{N}^\mathrm{sc}(\set^\mathrm{sc}_\Deltab)$.
\end{theorem}
\begin{proof}
Purely for notational convenience, we will prove that there is a map of $(\infty,2)$-categories from the $(\infty,2)$-category of coCartesian fibrations over a scaled simplicial set $(S^{op},\Gamma)$ to the $(\infty,2)$-category of functors from $(S^{op},\Gamma)$ to the $(\infty,2)$-category $\mathrm{N}^\mathrm{sc}(\set^\mathrm{sc}_\Deltab)$. Recall the ordinary straightening construction (in the unmarked case): Let $X$ be an object of $(\set_\Deltab)_{/S}$. Let $\cc$ be the simplicial category $\mathfrak{C}[S]^{op}$. We may then form another simplicial category, $\cc^{op}_X$, via the pushout:
\begin{equation*}
\xymatrix{
\mathfrak{C}[X]\ar[r]\ar[d] & \mathfrak{C}[X^\vartriangleright]\ar[d]\\
\cc^{op}\ar[r] & \cc^{op}_X
}
\end{equation*}
The map $\mathfrak{C}[X]\to \cc^{op}$ is denoted $\phi$. We can then define a functor $\mathrm{St}_\phi(X):\cc\to\set_\Deltab$ as follows: $\mathrm{St}_\phi(X)(C)=\mathrm{Map}_{\cc^{op}_X}(C,\ast)$ where $\ast$ is the cone point of $X^\vartriangleright$. This is functorial in $X$; therefore we may think of $\mathrm{St}_\phi$ as a functor $(\set_\Deltab)_{/S}\to(\set_\Deltab)^\cc$. Let $(X,\Gamma)$ be an object of $(\set^\mathrm{sc}_\Deltab)_{/S}$. Suppose $f:\Delta^2\to X$ is a $2$-simplex of $X$, and consider the induced map $f\star\mathrm{id}_{\Delta^0}:\Delta^3\to X^\vartriangleright$. This determines a map $\mathfrak{C}[\Delta^3]\to \cc_X$; explicitly, this construction is described as follows. Suppose we have a $2$-simplex of $X$:
\begin{equation*}
\xymatrix@C=1em@R=1em{
c \ar[dr]\ar[rr]& & \ar[dl]d\\
 & e &
}
\end{equation*}
Obviously, this corresponds to a $2$-simplex in $\cc^{op}$ under the map $\phi$. Taking the join with $\Delta^0$ gives us a $3$-simplex, which can be identified with a diagram:
\begin{equation*}
\xymatrix@C=1em@R=1em{
\phi(c) \ar[ddr]\ar[dr]\ar[rr]& & \ar[dl]\ar[ddl]\phi(d)\\
 & \ast &\\
 & \phi(e)\ar[u] &
}
\end{equation*}
Where $\ast$ is the cone point of $X^\vartriangleright$. This induces a $2$-simplex in $\mathrm{St}_\phi(X)(C)$; now one can say that a $2$-simplex in $\mathrm{St}_\phi(X)(C)$ is thin if it can be constructed in the above fashion from a thin $2$-simplex of $X$. The collection of thin $2$-simplices of $\mathrm{St}_\phi(X)(C)$ is denoted $\Gamma_\phi(C)$. This refines $\mathrm{St}_\phi$ to a functor $\mathrm{St}^\mathrm{sc}_\phi(X,\Gamma)(C) = (\mathrm{St}_\phi(X)(C),\Gamma_\phi(C))$; this is the desired functor.
\end{proof}
\begin{remark}
The functor $\mathrm{St}^\mathrm{sc}_\phi$ admits a right adjoint, denoted $\mathrm{Un}^\mathrm{sc}_\phi$.
\end{remark}
\section{$(\infty,2)$-operads}
\subsection{$\infty$-operads}
In this subsection, we will recall the theory of complete Segal operads via operator categories as defined by Barwick in \cite{operatorcat}. For the whole section, we will use the notation employed in \cite{operatorcat}, and we will also use $\mathrm{N}(\cc)$ to denote $\mathrm{N}^\mathrm{sc}(\cc)$.
\begin{definition}
An operator category $\Phi$ is a category that has a terminal object such that for every $X,Y\in\Phi$, the set $\mathrm{Map}_\Phi(X,Y)$ is finite, and for every $i:1\to K$ denoted as $\{i\}\subseteq K$ and every map $I\to K$, there is a fiber $I_i = I\times_K\{i\}$.
\end{definition}
The set $\mathrm{Map}_\Phi(1,K)$ is denoted $|K|$.
\begin{example}
Important examples include the trivial category $\{1\}$, the category of finite sets, $\cf$, and the category of ordered finite sets $\calO$. Given two operator categories $\Phi$ and $\Psi$, we can construct another operator category $\Phi\wr\Psi$ whose objects are pairs $(X,Y)$ with $X\in\Phi$ and $Y=\{Y_i\}_{i\in\mathrm{Map}_\Phi(1,X)}$ a collection of objects in $\Psi$, and morphisms $(X,Y)\to(X^\prime,Y^\prime)$ given by maps $f:X\to X^\prime$ and a collection of maps $\{Y_\tau\to Y^\prime_{f(\tau)}\}_{\tau\in\mathrm{Map}_\Phi(1,X)}$. This is called the wreath product of operator categories.
\end{example}
\begin{definition}
A map of operator categories is a functor $F:\Phi\to\Psi$ that preserves terminal objects and the formation of fibers such that for any $X\in\Phi$, the map $\mathrm{Map}_\Phi(1,X)\to\mathrm{Map}_\Psi(1,F(X))$ is a surjection.
\end{definition}
\begin{lemma}
The map $\mathrm{Map}_\Phi(1,X)\to\mathrm{Map}_\Psi(1,F(X))$ is a bijection.
\end{lemma}
\begin{proof}
This follows from there being an injection $\mathrm{Map}_\Phi(1,X)\to\mathrm{Map}_\Psi(1,F(X))$ because $F$ preserves fibers, and there are $|\mathrm{Map}_\Phi(1,X)|$ fibers of a map into $X$.
\end{proof}
As a consequence, $\cf$ is a terminal object, and $\{1\}$ is an initial object, in the $\infty$-category of operator categories.
\begin{definition}
A map $p:K\to J$ of operator categories is a fiber inclusion if there is a map $q:J\to I$ such that $I$ is the pullback $J\times_K\{i\}$. $p$ is an interval inclusion if it can be written as a composition of fiber inclusions.
\end{definition}
\begin{example}
Every monomorphism of finite sets is a fiber inclusion and conversely, i.e., fiber and interval inclusions coincide in $\cf$.
\end{example}
\begin{definition}
A $\Phi$-sequence is a pair $\mathbf{n}\in\Deltab$ and a functor $I:\mathbf{n}\to\Phi$, denoted by $[I_0\to\cdots\to I_n]$. A map of $\Phi$-sequences is a pair of maps map $(p,q):(\mathbf{n},I)\to(\mathbf{m},J)$ where $q$ is a natural transformation $q:I\to J\circ p$ such that for any $0\leq k\leq m$, the map $q_k:I_k\to J_{p(k)}$ is an interval inclusion, and for any $0\leq \sigma\leq \tau\leq m$, $J_{p(\tau)}\simeq J_{p(\sigma)}\times_{I_\sigma}I_\tau$. These form a category, denoted $\Deltab_\Phi$.
\end{definition}
\begin{definition}
Suppose $\mathbf{X}:\Deltab_\Phi^{op}\to\mathbf{Kan}$ is a functor. Define:
\begin{equation*}
p_{(\mathbf{n},I)}:\mathbf{X}[I_0\to I_1\to \cdots\to I_n]\to 
\prod_{t\in|I_n|} \mathbf{X}[I_{0,t}\to
I_{1,t}\to\cdots
\to I_{n,t}]
\end{equation*}
and:
\begin{equation*}
s_{(\mathbf{n},I)}:\mathbf{X}[I_0\to I_1\to \cdots\to I_n]\xrightarrow{}
\mathbf{X}[I_0\to I_1] \times_{\mathbf{X}[I_1]}\cdots\times_{\mathbf{X}[I_{n-1}]}
\mathbf{X}[I_{n-1}\to I_n]
\end{equation*}
Construct another map $r$ as follows. There is an inclusion $\Deltab\hookrightarrow\Deltab_{\Phi}$, so $\mathbf{X}$ restricts to a functor $\mathbf{X}|_{\Deltab^{op}}:\Deltab^{op}:(\mathrm{Set}^\mathrm{sc}_\Deltab)^\circ$. Denote by $K$ the simplicial set $\Delta^{3}/(\Delta^{\{0,2\}}\sqcup\Delta^{\{1,3\}})$. Let $\mathbf{X}|_{\Deltab^{op}_{K}}$ denote the homotopy limit of the diagram $\Deltab^{op}_{K}\to\Deltab^{op}\to\set^\mathrm{sc}_\Deltab$. Denote by $r$ the map $\mathbf{X}|_{\Deltab^{op}}\to \mathbf{X}|_{\Deltab^{op}_{K}}$.
\end{definition}
We can now finally define complete Segal $\Phi$-operads.
\begin{definition}
If $\Phi$ is an operator category, define a complete Segal $\Phi$-operad to be a left fibration $p:\mathbf{X}\to\mathrm{N}(\Deltab_\Phi)^{op}$ such that a functor $\Deltab_\Phi^{op}\to\mathbf{Kan}$ classifying $p$ has all the three maps $p_{(\mathbf{n},I)}$, $s_{(\mathbf{n},I)}$, and $r$ defined above as equivalences. A map of complete Segal $\Phi$-operads is a map in the overcategory $(\set_\Deltab)_{/\mathrm{N}(\Deltab_\Phi)^{op}}$.
\end{definition}
\begin{remark}
The fiber of the map $\mathbf{X}[I_\bullet] = \mathbf{X}[I_0\to I_1\to\cdots\to I_n]\to \mathbf{X}[I_0]\times\cdots\times \mathbf{X}[I_n]$ over the vertex $((x^0_{i_0})_{i_0\in |I_0|},\cdots,(x^n_{i_n})_{i_n\in |I_n|})$ is denoted $\mathrm{Map}_{\mathbf{X}^\otimes}^{[I_\bullet]}((x^0_{i_0})_{i_0\in |I_0|},\cdots,(x^n_{i_n})_{i_n\in |I_n|})$. Suppose $n=1$, and $I_n=\langle 1\rangle$; then $\mathrm{Map}_{\mathbf{X}^\otimes}^{[I_0\to I_1]}((x^0_{i_0})_{i_0\in |I_0|},y)$ can be thought of as the polycomposition maps from $(x^0_{i_0})_{i_0\in |I_0|}$ to $y$.
\end{remark}
The results we will state here are the following.
\begin{theorem}
Bousfield localizing the covariant model structure gives a left proper simplicial model structure on $(\set_\Deltab)_{/\mathrm{N}(\Deltab_\Phi)^{op}}$ such that:
\begin{itemize}
\item A map $X\to Y$ is a cofibration if it is a monomorphism.
\item A map $X\to Y$ of simplicial sets over $\mathrm{N}(\Deltab_\Phi)^{op}$ is a weak equivalence if for any complete Segal $\Phi$-operad $C$, the map $\mathrm{Map}_{\mathrm{N}(\Deltab_\Phi)^{op}}(Y,C)\to\mathrm{Map}_{\mathrm{N}(\Deltab_\Phi)^{op}}(X,C)$ is a weak equivalence.
\item An object is fibrant precisely if it is a complete Segal $\Phi$-operad.
\end{itemize}
In addition, if $F:\Phi\to\Psi$ is a map of operator categories, there is an adjunction $F_{!}:(\set_\Deltab)_{/\mathrm{N}(\Deltab_\Phi)^{op}}\leftrightarrow (\set_\Deltab)_{/\mathrm{N}(\Deltab_\Psi)^{op}}:F_{\star}$.
\end{theorem}
\begin{remark}\label{straightening}
There is a straightened version of this model structure. Namely, there is a left proper model structure on $\mathrm{Fun}(\Deltab_\Phi^{op},\set_\Deltab)$ such that natural transformations are cofibrations if they are monomorphisms, fibrant objects are those $\mathbf{Kan}$-valued maps which classifying complete Segal $\Phi$-operads, weak equivalences are natural transformations $X\to Y$ such that for any fibrant object $Z$, the induced map $\mathrm{Map}(Y,Z)\to \mathrm{Map}(X,Z)$ is a weak equivalence. There is a Quillen equivalence $\displaystyle (\set_\Deltab)_{/\mathrm{N}(\Deltab_\Phi)^{op}}\simeq_\mathrm{Q}\mathrm{Fun}(\Deltab_\Phi^{op},\set_\Deltab)$.
\end{remark}
If $\Psi = \cf$, the map $F_{!}$ takes the terminal $\Phi$-operad to its symmetrization.
\begin{theorem}
The symmetrization of the terminal $\calO^{\wr n}$-operad is the $\mathbf{E}_n$-operad. The symmetrization of the terminal $\cf$-operad, i.e., the terminal $\cf$-operad, is the $\mathbf{E}_\infty$-operad.
\end{theorem}
\begin{theorem}
For ``perfect'' operator categories $\Phi$, there is an equivalence of $\infty$-categories between complete Segal $\Phi$-operads and quasioperads over $\Phi$.
\end{theorem}
\begin{remark}\label{correspondence}
The most important case is when $\Phi=\cf$, in which case complete Segal $\cf$-operads correspond to Lurie's $\infty$-operads.
\end{remark}
\begin{remark}
Remark \ref{straightening} and Remark \ref{correspondence} give a version of straightening for Lurie's $\infty$-operads.
\end{remark}
\begin{example}
The following examples demonstrate the importance and generality of complete Segal $\Phi$-operads.
\begin{enumerate}
\item When $ \Phi=\{1\}$, the definition of a complete Segal $\Phi$-operad reduces to the definition of a complete Segal space.
\item When $ \Phi=\calO$, the definition of a complete Segal $\Phi$-operad reduces to the definition of a nonsymmetric $ \infty$-operad.
\item When $ \Phi=\cf$, the definition of a complete Segal $\Phi$-operad reduces to the definition of an $ \infty$-operad.
\end{enumerate}
\end{example}
\subsection{$2$-fold complete Segal operads}
Consider the $ (\infty,2)$-category of stable $\infty$-categories. This should be a symmetric monoidal $ (\infty,2)$-category with the tensor product of stable $ \infty$-categories. It makes sense to ask for a more general notion of a symmetric monoidal $(\infty,2)$-category, namely the theory of ``$ (\infty,2)$-operads''.
\begin{construction}
Let $\Phi$ be an operator category. Define the category of $(\Phi,2)$-sequences $\Deltab_\Phi\boxtimes\Deltab_\Phi$ as the category whose objects are pairs $((\mathbf{n},I:\mathbf{n}\to\Phi),\{([|I_k|],J^k:[|I_k|]\to\Phi)\}_{k\in\mathbf{n}})$ and a morphism $((\mathbf{n},I:\mathbf{n}\to\Phi),\{([|I_k|],J^k:[|I_k|]\to\Phi)\}_{k\in\mathbf{n}})\to ((\mathbf{m},I:\mathbf{m}\to\Phi),\{([|I_g|],J^g:[|I_g|]\to\Phi)\}_{g\in\mathbf{m}})$ are maps of the corresponding $\Phi$-sequences.
\end{construction}
\begin{remark}
A $(\Phi,2)$-sequence is to be thought of as a sequence $[I_0\xrightarrow{J^0_0\to\cdots\to J^0_m}I_1\xrightarrow{J^1_0\to\cdots\to J^1_\ell}\cdots\xrightarrow{J^{n-1}_0\to\cdots\to J^{n-1}_{m^\prime}}I_n]$.
\end{remark}
\begin{construction}
Let $\mathbf{X}:(\Deltab_\Phi\boxtimes\Deltab_\Phi)^{op}\to\set^\mathrm{sc}_\Deltab$ be a functor, and let $\overline{\mathbf{X}}$ denote the restriction of $\mathbf{X}$ to the second factor. Define the map $p_{s,I}$ as:
\begin{multline*}
\mathbf{X}[I_0\xrightarrow{J^0_0\to\cdots\to J^0_m}I_1\xrightarrow{J^1_0\to\cdots\to J^1_\ell}\cdots\xrightarrow{J^{n-1}_0\to\cdots\to J^{n-1}_{m^\prime}}I_n]\to 
\prod_{t\in|I_n|} \mathbf{X}[I_{0,t}\xrightarrow{\{t\}\times\prod_{i\in|J^0_m|}\overline{\mathbf{X}}[J^0_{0,i}\to\cdots\to J^0_{m,i}]}\\
I_{1,t}\xrightarrow{\{t\}\times\prod_{q\in|J^1_\ell|}\overline{\mathbf{X}}[J^1_{0,q}\to\cdots \to J^1_{\ell,q}]}\cdots
\xrightarrow{\{t\}\times\prod_{k\in|J^{n-1}_{m^\prime}|}\overline{\mathbf{X}}[J^{n-1}_{0,k}\to\cdots\to J^{n-1}_{m^\prime,k}]}I_{n,t}]
\end{multline*}
\end{construction}
\begin{construction}
As above, let $\mathbf{X}:(\Deltab_\Phi\boxtimes\Deltab_\Phi)^{op}\to\set^\mathrm{sc}_\Deltab$ be a functor, and let $\overline{\mathbf{X}}$ denote the restriction of $\mathbf{X}$ to the second factor. Define the map $s_{s,I}$ as:
\begin{multline*}
\mathbf{X}[I_0\xrightarrow{J^0_0\to\cdots\to J^0_m}I_1\xrightarrow{J^1_0\to\cdots\to J^1_\ell}\cdots\xrightarrow{J^{n-1}_0\to\cdots\to J^{n-1}_{m^\prime}}I_n]\xrightarrow{s_{s,I}}\\
\mathbf{X}[I_0\xrightarrow{\overline{\mathbf{X}}[J^0_0\to J^0_1]\times_{\overline{\mathbf{X}}[J^0_1]}\cdots\times_{\overline{\mathbf{X}}[J^0_{m-1}]}\overline{\mathbf{X}}[J^0_{m-1}\to J^0_m]}I_1] \times_{\mathbf{X}[I_1]}\cdots\times_{\mathbf{X}[I_{n-1}]}
\\ \mathbf{X}[I_{n-1}\xrightarrow{\overline{\mathbf{X}}[J^{n-1}_0\to J^{n-1}_1]\times_{\overline{\mathbf{X}}[J^{n-1}_1]}\cdots\times_{\overline{\mathbf{X}}[J^{n-1}_{m^\prime -1}]}\overline{\mathbf{X}}[J^{n-1}_{m^\prime -1}\to J^{n-1}_{m^\prime}]}I_n]
\end{multline*}
\end{construction}
\begin{construction}
As above, let $\mathbf{X}:(\Deltab_\Phi\boxtimes\Deltab_\Phi)^{op}\to\set^\mathrm{sc}_\Deltab$ be a functor. There is an inclusion $\Deltab^{\wr 2}\hookrightarrow\Deltab_\Phi\boxtimes\Deltab_\Phi$, so $\mathbf{X}$ restricts to a functor $X|_{(\Deltab^{\wr 2})^{op}}:(\Deltab^{\wr 2})^{op}:(\mathrm{Set}^\mathrm{sc}_\Deltab)^\circ$. Denote by $K$ the simplicial set $\Delta^{3}/(\Delta^{\{0,2\}}\sqcup\Delta^{\{1,3\}})$. Let $\mathbf{X}|_{(\Deltab^{\wr 2})^{op}_{K}}$ denote the homotopy limit of the diagram $(\Deltab^{\wr 2})^{op}_{K}\to(\Deltab^{\wr 2})^{op}\to\set^\mathrm{sc}_\Deltab$. Denote by $r_2$ the map $\mathbf{X}|_{(\Deltab^{\wr 2})^{op}}\to \mathbf{X}|_{(\Deltab^{\wr 2})^{op}_{K}}$.
\end{construction}
\begin{remark}
These are the ``right maps'' because under the (limit-preserving) canonical forgetful functor $(\Deltab_\Phi\boxtimes\Deltab_\Phi)^{op}\xrightarrow{F}\Deltab^{op}_\Phi$ these reduce to the maps defined by Barwick in \cite{operatorcat}.
\end{remark}
Let $\mathbf{X}^\otimes$ be a scaled simplicial set equipped with a left $2$-fibration $p:\mathbf{X}^\otimes\to\mathrm{N}(\Deltab_\Phi\boxtimes\Deltab_\Phi)^{op}$. This corresponds, via straightening, to a map $(\Deltab_\Phi\boxtimes\Deltab_\Phi)^{op}\to\set^\mathrm{sc}_\Deltab$. We can now define $2$-fold complete Segal $\Phi$-operads.
\begin{definition}
A $2$-fold complete Segal $\Phi$-operad is a left $2$-fibration $p:\mathbf{X}^\otimes\to\mathrm{N}(\Deltab_\Phi\boxtimes\Deltab_\Phi)^{op}$ such that if $\chi$ is its classifying map, then $p_{s,I},s_{s,I}$, and $r_2$ are all equivalences.
\end{definition}
\begin{remark}
Since $p:\mathbf{X}^\otimes\to\mathrm{N}(\Deltab^{\wr 2}_\Phi)^{op}$ is a left $2$-fibration, when $\Phi=\{1\}$, this reduces to the notion of a $2$-fold complete Segal space.
\end{remark}
\begin{example}\label{inclusionoperads}
This definition yields, when $\Phi=\cf$, a theory of $(\infty,2)$-operads (since $p$ is a left $2$-fibration, the fibers of $p$ over $s\in\mathrm{N}(\Deltab_\Phi\boxtimes\Deltab_\Phi)$ are all $\infty$-categories. This means that $X$ also contains information about the operadic structure at the level of higher homotopies as well). In particular, any left fibration is a left $2$-fibration, so any $\infty$-operad is trivially an $(\infty,2)$-operad. More generally, any complete Segal $\Phi$-operad is a $2$-fold complete Segal $\Phi$-operad.
\end{example}
\begin{example}\label{symmon}
By definition, a symmetric monoidal $(\infty,2)$-category is a section of the the functor $\mathrm{Cat}_{(\infty,2)}\to \mathrm{N}(\Deltab^{op}_\Phi)$ where $\mathrm{Cat}_{(\infty,2)}$ is the $\infty$-category of $(\infty,2)$-categories and $\Phi=\cf$, i.e., a functor $\Deltab^{op}_\Phi\to(\mathrm{Set}^\mathrm{sc}_\Deltab)^\circ$, which can be extended to a map $\Deltab^{op}_\Phi\to\set^\mathrm{sc}_\Deltab$. There is a canonical forgetful functor $(\Deltab_\Phi\boxtimes\Deltab_\Phi)^{op}\xrightarrow{F}\Deltab^{op}_\Phi$, and therefore a symmetric monoidal $(\infty,2)$-category is a functor $(\Deltab_\Phi\boxtimes\Deltab_\Phi)^{op}\xrightarrow{F}\Deltab^{op}_\Phi\to(\mathrm{Set}^\mathrm{sc}_\Deltab)^\circ\to\set^\mathrm{sc}_\Deltab$. Since $F$ is a forgetful functor, $p_{s,I},s_{s,I}$, and $r_2$ are all equivalences, and hence any symmetric monoidal $(\infty,2)$-category is an $(\infty,2)$-operad.
\end{example}
Choose a scaled simplicial set $S$. We observe that left $2$-fibrations completely determine fibrant objects of the left model structure (Theorem \ref{leftmodel}) on $(\mathrm{Set}^\mathrm{sc}_\Deltab)_{/S}$ in the sense that a left $2$-fibration $X\to S$ with base $S$ determines a fibrant scaled simplicial set $X^\natural$. A suitable Bousfield localization of $(\mathrm{Set}^\mathrm{sc}_\Deltab)_{/\mathrm{N}(\Deltab_\Phi\boxtimes\Deltab_\Phi)^{op}}$ can be constructed such that any fibrant object of $(\mathrm{Set}^\mathrm{sc}_\Deltab)_{/\mathrm{N}(\Deltab_\Phi\boxtimes\Deltab_\Phi)^{op}}$ yields a $2$-fold complete Segal $\Phi$-operad.
\begin{theorem}[$\Phi$-operadic model structure]
There is a combinatorial left proper simplicial model structure on $(\mathrm{Set}^\mathrm{sc}_\Deltab)_{/\mathrm{N}(\Deltab_\Phi\boxtimes\Deltab_\Phi)^{op}}$ that is a localization of the left model structure on $(\mathrm{Set}^\mathrm{sc}_\Deltab)_{/\mathrm{N}(\Deltab_\Phi\boxtimes\Deltab_\Phi)^{op}}$ such that every fibrant object $X\simeq (\mathbf{Y}^\otimes)^\natural$ of $(\mathrm{Set}^\mathrm{sc}_\Deltab)_{/\mathrm{N}(\Deltab_\Phi\boxtimes\Deltab_\Phi)^{op}}$ determines a $2$-fold complete Segal $\Phi$-operad $p:\mathbf{Y}^\otimes\to{\mathrm{N}(\Deltab_\Phi\boxtimes\Deltab_\Phi)^{op}}$.
\end{theorem}
\begin{construction}
If $F:\Phi\to \Psi$ is a map of operator categories, there is an induced map $\mathrm{N}(\Deltab_\Phi\boxtimes\Deltab_\Phi)^{op}\to\mathrm{N}(\Deltab^{\wr 2}_{\Psi})^{op}$, which induces an adjunction $F_{!}:(\mathrm{Set}^\mathrm{sc}_\Deltab)_{/\mathrm{N}(\Deltab_\Phi\boxtimes\Deltab_\Phi)^{op}}\leftrightarrow(\mathrm{Set}^\mathrm{sc}_\Deltab)_{/\mathrm{N}(\Deltab^{\wr 2}_{\Psi})^{op}}:F_\star$. Let $u:\Phi\to \cf$ be the essentially unique map of operator categories. The symmetrization of a $2$-fold complete Segal $\Phi$-operad $\mathbf{X}^\otimes\to{\mathrm{N}(\Deltab_\Phi\boxtimes\Deltab_\Phi)^{op}}$ is the image of $\mathbf{X}^\otimes\to{\mathrm{N}(\Deltab_\Phi\boxtimes\Deltab_\Phi)^{op}}$ under $u_{!}$. 
\end{construction}
\begin{definition}
Let $\mathrm{id}_{\mathrm{N}(\Deltab_\Phi\boxtimes\Deltab_\Phi)^{op}}:{\mathrm{N}(\Deltab_\Phi\boxtimes\Deltab_\Phi)^{op}}\to{\mathrm{N}(\Deltab_\Phi\boxtimes\Deltab_\Phi)^{op}}$ denote the terminal object of $(\mathrm{Set}^\mathrm{sc}_\Deltab)_{/\mathrm{N}(\Deltab_\Phi\boxtimes\Deltab_\Phi)^{op}}$. Then:
\begin{enumerate}
\item The symmetrization of $\mathrm{id}_{\mathrm{N}(\Deltab_\Phi\boxtimes\Deltab_\Phi)^{op}}$ when $\Phi=\cf$ is denoted $\mathbf{E}^2_\infty$.
\item The symmetrization of $\mathrm{id}_{\mathrm{N}(\Deltab_\Phi\boxtimes\Deltab_\Phi)^{op}}$ when $\Phi=\calO$ is the operator category of ordered finite sets is denoted $\mathbf{E}^2_1$.
\item For an integer $k\geq 0$, the symmetrization of $\mathrm{id}_{\mathrm{N}(\Deltab_\Phi\boxtimes\Deltab_\Phi)^{op}}$ when $\Phi=\calO^{(k)}=\calO\wr\calO\cdots\wr\calO$ is the wreath product of $\calO$ with itself $k$ times is denoted $\mathbf{E}^2_k$.
\end{enumerate}
\end{definition}
These are analogues of the commutative, associative, and $\mathbf{E}_k$-operads.
\subsection{$2$-rings}
\begin{definition}
Throughout the rest of this section, fix $\Phi=\cf$. Let $\cc$ be a symmetric monoidal $(\infty,2)$-category, and let $p:\mathbf{Y}^\otimes\to{\mathrm{N}(\Deltab_\Phi\boxtimes\Deltab_\Phi)^{op}}$ be a complete Segal $\Phi$-operad. A $\mathbf{Y}$-algebra object of $\cc$ is a map of fibrant objects $\mathbf{Y}^\otimes\to \cc$ over ${\mathrm{N}(\Deltab_\Phi\boxtimes\Deltab_\Phi)^{op}}$ in $(\mathrm{Set}^\mathrm{sc}_\Deltab)_{/\mathrm{N}(\Deltab_\Phi\boxtimes\Deltab_\Phi)^{op}}$.
\end{definition}
\begin{example}
If $\mathbf{Y}^\otimes=\mathbf{E}^2_\infty$ and $\cc$ is $\mathrm{Cat}_\infty$, this yields a weaker notion of symmetric monoidal $\infty$-category. 
\end{example}
\begin{construction}\label{forgetful}
Consider the identity map $\mathrm{N}(\Deltab_\Phi\boxtimes\Deltab_\Phi)^{op}\to\mathrm{N}(\Deltab_\Phi\boxtimes\Deltab_\Phi)^{op}$. Define a map of scaled simplicial sets $\mathrm{N}(\Deltab_\Phi)^{op}\to \mathrm{N}(\Deltab_\Phi\boxtimes\Deltab_\Phi)^{op}$ which takes $(\mathbf{n},I:\mathbf{n}\to \Phi)$ to $((\mathbf{n},I:\mathbf{n}\to\Phi),\{([|I_k|],J^k:[|I_k|]\to\Phi)\}_{k\in\mathbf{n}})$ with $J^k$ the constant $\Phi$-sequence of length $|I_k|$ where every element is $I_k$. There is a canonical map $\mathrm{N}(\Deltab_\Phi\boxtimes\Deltab_\Phi)^{op}\to\mathrm{N}(\Deltab_\Phi)^{op}$ given by taking $((\mathbf{n},I:\mathbf{n}\to\Phi),\{([|I_k|],J^k:[|I_k|]\to\Phi)\}_{k\in\mathbf{n}})$ to $(\mathbf{n},I:\mathbf{n}\to \Phi)$. The composite $\mathrm{N}(\Deltab_\Phi)^{op}\to\mathrm{N}(\Deltab_\Phi)^{op}$ is the identity map, and when $\Phi=\cf$, we can consider the symmetrization of the underlying map of simplicial sets; this is the $\infty$-operad $\mathbf{E}_\infty$. We therefore get a map $\mathbf{E}^2_\infty\to \overline{\mathbf{E}_\infty}$ where $\overline{\mathbf{E}_\infty}$ is the $\mathbf{E}_\infty$-operad viewed as an $(\infty,2)$-operad (Example \ref{inclusionoperads}).
\end{construction}
\begin{prop}\label{reduction}
Let $\cc$ be a symmetric monoidal $\infty$-category viewed as a symmetric monoidal $(\infty,2)$-category by regarding each mapping simplicial set as a marked simplicial set. Every $\mathbf{E}_\infty$-algebra object of $\cc$ is a $\mathbf{E}^2_\infty$-algebra object of $\cc$.
\end{prop}
\begin{proof}
Let $\overline{\mathbf{E}_\infty}$ denote the $\mathbf{E}_\infty$-operad viewed as an $(\infty,2)$-operad (Example \ref{inclusionoperads}) by regarding each mapping simplicial set as a marked simplicial set. By abuse of notation we will write $\cc$ for $\cc$ regarded as a symmetric monoidal $(\infty,2)$-category. Suppose the map $\mathbf{E}^2_\infty\to\cc$ factors as $\mathbf{E}^2_\infty\to \overline{\mathbf{E}_\infty}\to\cc$ over ${\mathrm{N}(\Deltab_\Phi\boxtimes\Deltab_\Phi)^{op}}$ in $(\mathrm{Set}^\mathrm{sc}_\Deltab)_{/\mathrm{N}(\Deltab_\Phi\boxtimes\Deltab_\Phi)^{op}}$, where the map $\mathbf{E}^2_\infty\to \mathbf{E}_\infty$ is constructed above. Since the map $\mathbf{E}^2_\infty\to \mathbf{E}_\infty$ is uniquely specified, such an $\mathbf{E}^2_\infty$-algebra object of $\cc$ is equivalent to a map $\overline{\mathbf{E}_\infty}\to\cc$ over ${\mathrm{N}(\Deltab_\Phi\boxtimes\Deltab_\Phi)^{op}}$ in $(\mathrm{Set}^\mathrm{sc}_\Deltab)_{/\mathrm{N}(\Deltab_\Phi\boxtimes\Deltab_\Phi)^{op}}$. This is equivalent to a map $\mathbf{E}_\infty\to\cc$ of ordinary $\infty$-operads by the assumptions on $\cc$; but this is just the definition of an $\mathbf{E}_\infty$-algebra object of $\cc$, so we are done.
\end{proof}
\begin{corollary}\label{consistency}
Every symmetric monoidal $\infty$-category (resp. symmetric monoidal stable $\infty$-category) is a $\mathbf{E}^2_\infty$-algebra object of $\mathrm{Cat}_\infty$ (resp. $\mathrm{Cat}^\mathrm{st}_\infty$).
\end{corollary}
\begin{definition}
A $2$-ring is a $\mathbf{E}^2_\infty$-algebra object of $\mathbf{Mod}_\mathbf{Sp}(\cP\mathrm{r^L})$.
\end{definition}
\begin{example}
By Corollary \ref{consistency}, any symmetric monoidal stable $\infty$-category is a $2$-ring. Thus, if $R$ is a $\mathbf{E}_\infty$-ring, the stable $\infty$-category $\mathbf{Mod}_R$ is a $2$-ring.
\end{example}
\begin{construction}
Let $p:\mathbf{X}^\otimes\to{\mathrm{N}(\Deltab_\Phi\boxtimes\Deltab_\Phi)^{op}}$ be a $2$-fold complete Segal $\Phi$-operad, and let $\mathbf{X}:{\mathrm{N}(\Deltab_\Phi\boxtimes\Deltab_\Phi)^{op}}\to\mathrm{N}^\mathrm{sc}(\set^\mathrm{sc}_\Deltab)$ be its classifying map. The maximal $\infty$-category in the fiber of the map $\mathbf{X}[I_\bullet] = \mathbf{X}[I_0\xrightarrow{J^0_0\to\cdots\to J^0_m}I_1\xrightarrow{J^1_0\to\cdots\to J^1_\ell}\cdots\xrightarrow{J^{n-1}_0\to\cdots\to J^{n-1}_{m^\prime}}I_n]\to \mathbf{X}[I_0]\times\cdots\times \mathbf{X}[I_n]$ over the vertex $((x^0_{i_0})_{i_0\in |I_0|},\cdots,(x^n_{i_n})_{i_n\in |I_n|})$ is denoted $\mathrm{Map}_{\mathbf{X}^\otimes}^{[I_\bullet]}((x^0_{i_0})_{i_0\in |I_0|},\cdots,(x^n_{i_n})_{i_n\in |I_n|})$. Suppose $n=1$, and $I_n=\langle 1\rangle$; then $\mathrm{Map}_{\mathbf{X}^\otimes}^{[I_0\xrightarrow{J^0_\bullet}I_1]}((x^0_{i_0})_{i_0\in |I_0|},y)$ can be thought of as the polycomposition maps from $(x^0_{i_0})_{i_0\in |I_0|}$ to $y$. If $\mathbf{X}^\otimes$ is a symmetric monoidal $(\infty,2)$-category, $(x^0_{i_0})_{i_0\in |I_0|}$ is denoted $\bigotimes_{i_0\in |I_0|} x^0_{i_0}$, and is called the tensor product of the $x^0_{i_0}$.
\end{construction}
\begin{remark}
This is technically abuse of notation since the fiber of the map $\mathbf{X}[I_\bullet] = \mathbf{X}[I_0\xrightarrow{J^0_0\to\cdots\to J^0_m}I_1\xrightarrow{J^1_0\to\cdots\to J^1_\ell}\cdots\xrightarrow{J^{n-1}_0\to\cdots\to J^{n-1}_{m^\prime}}I_n]\to \mathbf{X}[I_0]\times\cdots\times \mathbf{X}[I_n]$ over the vertex $((x^0_{i_0})_{i_0\in |I_0|},\cdots,(x^n_{i_n})_{i_n\in |I_n|})$ depends on the choices of $J^i_\bullet$. If $n=1$, and $I_n=\langle 1\rangle$, for the rest of this paper, we will let $J^0_\bullet = \langle 1\rangle$.
\end{remark}
\begin{example}\label{tensorexample}
If $\mathbf{X}^\otimes$ is an ordinary complete Segal $\Phi$-operad viewed as a $2$-fold complete Segal $\Phi$-operad, this reduces to the ordinary notion of a polycomposition map as in \cite{operatorcat}, because in this case, the maximal $\infty$-category in the fiber is simply the fiber itself. If $\mathbf{X}^\otimes$ is also a symmetric monoidal $\infty$-category, this reduces to the ordinary tensor product via Example \ref{inclusionoperads} and Example \ref{symmon}.
\end{example}
\begin{definition}
Let $\cc$ be a symmetric monoidal $(\infty,2)$-category, and let $M$ be an object of $\cc$. An endomorphism object $\mathrm{End}_{\cc}(M)$ of $M$ is an object representing the functor $\cc\to \mathrm{N}(\set^+_\Deltab)^\circ$ given by $C\mapsto \mathrm{Map}_{\cc^\otimes}(C\otimes M,M)$.
\end{definition}
\begin{prop}
Let $\cc$ be a $2$-ring. The algebraic K-theory $K(\cc)$ is an $\mathbf{E}_\infty$-ring.
\end{prop}
\begin{proof}
$K(\cc)$ is a $\mathbf{E}^2_\infty$-algebra object of $\mathbf{Sp}^\sharp$. Since $\mathbf{Sp}^\sharp = \mathrm{N}^\mathrm{sc}(\overline{\mathbf{Sp}})$, where $\overline{\mathbf{Sp}}$ is the $\set^+_\Deltab$-enriched category where every mapping simplicial set is regarded as a marked simplicial set, it follows from Proposition \ref{reduction} that this corresponds to a $\mathbf{E}_\infty$-ring.
\end{proof}
\begin{remark}
There is a quadfiltration of $\mathbf{E}^2_\infty$, as in the following diagram.
\begin{equation*}
\xymatrix@!0{
\mathbf{E}^2_0\ar@[red][rr]\ar@{=}@[red][dd]\ar[dr] & & \mathcal{P}(1,1)\ar@[red][rr]\ar@[red][dd]\ar[dr] & & \cdots\ar@[red][rr]\ar@[red][dd]\ar[dr] & & \mathcal{P}(n,1)\ar@[red][rr]\ar@[red][dd]\ar[dr] & & \cdots\ar@[red][rr]\ar@[red][dd]\ar[dr] & & \mathcal{P}(\infty,1)\ar@[red][dd]\ar[dr]
\\
 & \mathbf{E}_0\ar@[blue]'[r][rr]\ar@{=}@[blue]'[d][dd] & & A^1_1\ar@[blue]'[r][rr]\ar@[blue]'[d][dd] & & \cdots\ar@[blue]'[r][rr]\ar'[d][dd] & & A^n_1\ar@[blue]'[r][rr]\ar@[blue]'[d][dd] & & \cdots\ar@[blue]'[r][rr]\ar@[blue]'[d][dd] & & A^\infty_1\ar@[blue][dd]
\\
\vdots\ar@[red][rr]\ar@{=}@[red][dd]\ar[dr] & & \vdots\ar@[red][rr]\ar@[red][dd]\ar[dr] & & \vdots\ar@[red][rr]\ar@[red][dd]\ar[dr] & & \vdots\ar@[red][rr]\ar@[red][dd]\ar[dr] & & \vdots\ar@[red][rr]\ar@[red][dd]\ar[dr] & & \vdots\ar[dd]\ar[dr]
\\
 & \vdots\ar@[blue]'[r][rr]\ar@{=}@[blue]'[d][dd] & & \vdots\ar@[blue]'[r][rr]\ar@[blue]'[d][dd] & & \vdots\ar@[blue]'[r][rr]\ar@[blue]'[d][dd] & & \vdots\ar@[blue]'[r][rr]\ar@[blue]'[d][dd] & & \vdots\ar@[blue]'[r][rr]\ar@[blue]'[d][dd] & & \vdots\ar@[blue][dd]
\\
\mathbf{E}^2_0\ar@[red][rr]\ar@{=}@[red][dd]\ar[dr] & & \mathcal{P}(1,m)\ar@[red][rr]\ar@[red][dd]\ar[dr] & & \cdots\ar@[red][rr]\ar@[red][dd]\ar[dr] & & \mathcal{P}(n,m)\ar@[red][rr]\ar@[red][dd]\ar[dr] & & \cdots\ar@[red][rr]\ar@[red][dd]\ar[dr] & & \mathcal{P}(\infty,m)\ar@[red][dd]\ar[dr]
\\
 & \mathbf{E}_0\ar@[blue]'[r][rr]\ar@{=}@[blue]'[d][dd] & & A^1_m\ar@[blue]'[r][rr]\ar@[blue]'[d][dd] & & \cdots\ar@[blue]'[r][rr]\ar@[blue]'[d][dd] & & A^n_m\ar@[blue]'[r][rr]\ar@[blue]'[d][dd] & & \cdots\ar@[blue]'[r][rr]\ar@[blue]'[d][dd] & & A^\infty_m\ar@[blue][dd]
\\
\vdots\ar@[red][rr]\ar@{=}@[red][dd]\ar[dr] & & \vdots\ar@[red][rr]\ar@[red][dd]\ar[dr] & & \vdots\ar@[red][rr]\ar@[red][dd]\ar[dr] & & \vdots\ar@[red][rr]\ar@[red][dd]\ar[dr] & & \vdots\ar@[red][rr]\ar@[red][dd]\ar[dr] & & \vdots\ar@[red][dd]\ar[dr]
\\
 & \vdots\ar@[blue]'[r][rr]\ar@{=}@[blue]'[d][dd] & & \vdots\ar@[blue]'[r][rr]\ar@[blue]'[d][dd] & & \vdots\ar@[blue]'[r][rr]\ar@[blue]'[d][dd] & & \vdots\ar@[blue]'[r][rr]\ar@[blue]'[d][dd] & & \vdots\ar@[blue]'[r][rr]\ar@[blue]'[d][dd] & & \vdots\ar@[blue][dd]
\\
\mathbf{E}^2_0\ar@[red][rr]\ar[dr] & & \mathbf{E}^2_1\ar@[red][rr]\ar[dr] & & \cdots\ar@[red][rr]\ar[dr] & & \mathbf{E}^2_m\ar@[red][rr]\ar[dr] & & \cdots\ar@[red][rr]\ar[dr] & & \mathbf{E}^2_\infty\ar[dr]
\\
 & \mathbf{E}_0\ar@[blue][rr] & & A^1_1\ar@[blue][rr] & & \cdots\ar@[blue][rr] & & A^n_1\ar@[blue][rr] & & \cdots\ar@[blue][rr] & & A^\infty_1
}
\end{equation*}
The blue arrows give the bifiltration of $\mathbf{E}_\infty$ as in \cite[Example 11.5]{operatorcat}, and the red arrows give the bifiltration of $\mathbf{E}^2_\infty$ in terms of $(\infty,2)$-operads that are not $\infty$-operads. This might give an obstruction theory for $\mathbf{E}^2_\infty$-structures on objects of symmetric monoidal $(\infty,2)$-categories, and also an obstruction theory for extending $\mathbf{E}^2_k$-structures to $\mathbf{E}_k$-structures. We have not explored this.
\end{remark}
\begin{definition}\label{modules}
Let $\cc$ be a symmetric monoidal $(\infty,2)$-category, and let $R$ be a $\mathbf{E}^2_\infty$-algebra object of $\cc$. An $R$-module is an object $M$ of $\cc$ with a map $R\to \mathrm{End}_{\cc}(M)$.
\end{definition}
\begin{remark}
By the definition of an endomorphism object, an $R$-module $M$ is equivalently an object of $\cc$ with a map $R\otimes M\to M$. Since we assumed that $J^0_\bullet = \langle 1\rangle$, if $\cc$ is a symmetric monoidal $\infty$-category viewed as a symmetric monoidal $(\infty,2)$-category, this corresponds precisely to the tensor product in the $\infty$-category $\cc$.
\end{remark}
\begin{prop}
There is a universal property for the tensor product, and this uniquely determines the tensor product. In other words, if $I_\bullet = I_0\xrightarrow{J^0_\bullet} I_1\xrightarrow{J^1_\bullet} I_2$:
\begin{eqnarray*}
\mathrm{Map}_{\mathbf{X}^\otimes}^{[I_\bullet]}((x^0_{i_0})_{i_0\in |I_0|},\mathrm{Map}_{\mathbf{X}^\otimes}^{[I^\prime_\bullet]}((x^1_{i_1})_{i_1\in |I_1|},(x^2_{i_2})_{i_2\in |I_2|}))\simeq\\
\mathrm{Map}_{\mathbf{X}^\otimes}^{[I_\bullet]}((x^0_{\min |I_0|},\cdots,x^0_{\max |I_0|},x^1_{\min |I_1|},\cdots,x^1_{\max |I_1|}),(x^2_{i_2})_{i_2\in |I_2|})
\end{eqnarray*}
where $I^\prime_\bullet$ is the $(\Phi,2)$-sequence obtained by removing $I^\prime_0$.
\end{prop}
\begin{proof}
The first claim follows from diagram-chasing in the following commutative diagram:
\begin{equation*}
\xymatrix{\mathbf{X}[I_\bullet]\ar[r]\ar@{=}[d] & \mathbf{X}[I_0]\times \mathbf{X}[I^\prime_\bullet]\ar[d]\\
\mathbf{X}[I_\bullet]\ar[r] & \mathbf{X}[I_0]\times \mathbf{X}[I_1]\times \mathbf{X}[I_2]}
\end{equation*}
Note the dependence on the $J^\ast_\bullet$; if they differed from point-to-point, then the diagram would not even commute. To see that this uniquely specifies the tensor product, it suffices to note that it does so in the case $I_\bullet = I_0\xrightarrow{J^0_\bullet} I_1\xrightarrow{J^1_\bullet} I_2$, because in the general case, we get a commutative diagram:
\begin{equation*}
\xymatrix{\mathbf{X}[I_\bullet]\ar[r]\ar@{=}[d] & \mathbf{X}[I_0]\times \mathbf{X}[I^\prime_\bullet]\ar[r] & \cdots\ar[r] & \mathbf{X}[I_0]\times\cdots\times \mathbf{X}[I_{n-2}]\times \mathbf{X}[I^{\prime\cdots\prime}_\bullet]\ar[d]\\
\mathbf{X}[I_\bullet]\ar[rrr] & & & \mathbf{X}[I_0]\times \mathbf{X}[I_1]\times \cdots\times \mathbf{X}[I_n]}
\end{equation*}
\end{proof}
\begin{example}\label{modulesoverordcats}
Fix some $J^0_\bullet$, not necessarily $\langle 1\rangle$. By Example \ref{tensorexample}, if $\cc$ is a symmetric monoidal stable $\infty$-category, any $\cc$-module is a $\cc$-module in the sense of Definition \ref{modules}.
\end{example}
\begin{definition}
Again, fix some $J^0_\bullet$, not necessarily $\langle 1\rangle$. An ideal of a $2$-ring $\cc$ is a submodule of $\cc$, when viewed as a module over itself.
\end{definition}
\begin{example}
It follows from Example \ref{modulesoverordcats} that if $J^0_\bullet = \langle 1\rangle$, every thick subcategory of $\mathbf{Sp}$ is an ideal of $\mathbf{Sp}^\sharp$, and vice-versa.
\end{example}

\end{document}